\documentclass[12pt,draft]{amsart}
\usepackage[cp1251]{inputenc}
\usepackage[T2A]{fontenc}
\usepackage[english]{babel}
\usepackage{amsmath,amsfonts,amssymb}
\usepackage{geometry}

\newtheorem{lemma}{Lemma}
\newtheorem{proposition}{Proposition}

\theoremstyle{remark}

\theoremstyle{definition}

\begin{document}

\title[Parameter--elliptic operators on the extended Sobolev scale]
{Parameter--elliptic operators \\on the extended Sobolev scale}

\author[A. Murach]{Aleksandr A. Murach}

\address{Institute of Mathematics, National Academy of Sciences of Ukraine,
3 Tere\-shchen\-kivs'ka, Kyiv, 01601, Ukraine}

\email{murach@imath.kiev.ua}

\thanks{This research was partly supported by grant no. 01/01.12 of
National Academy of Sciences of Ukra\-ine (under the joint Ukrainian--Russian project
of NAS of Ukraine and Russian Foundation of Basic Research).}


\author[T. Zinchenko]{Tetiana Zinchenko}

\address{Institute of Mathematics, National Academy of Sciences of Ukraine,
3 Tere\-shchen\-kivs'ka, Kyiv, 01601, Ukraine}

\email{tzinchenko@imath.kiev.ua}




\begin{abstract}
Parameter--elliptic pseudodifferential operators given on a closed smooth manifold
are investigated on the extended Sobolev scale. This scale consists of all Hilbert
spaces that are interpolation spaces with respect to the Hilbert Sobolev scale. We
prove that these operators set isomorphisms between appropriate spaces of the scale
provided the parameter is modulo large enough. For solutions to the corresponding
parameter--elliptic equations, we establish two-sided a priori estimates, in which
the constants are independent of the parameter.
\end{abstract}

\maketitle

\section{Introduction}\label{sec1}

Parameter--elliptic operators occupy a special position in the theory of elliptic
differential equations. These operators are distinguished by the following
fundamental property: if the complex parameter is modulo large enough, then the
elliptic operator defines an isomorphism between appropriate Sobolev spaces, and
moreover the solution of the ellip\-tic equation admits an a priory estimate in
which the constant does not depend on the parameter. Elliptic operators with
spectral parameter are simple and important examples of the operators discussed.
Various classes of parameter-elliptic equations and boundary--value problems were
introduced and investigated in the papers by S.~Agmon \cite{Agmon62}, S.~Agmon and
L.~Nirenberg \cite{AgmonNirenberg63}, M.~S.~Agranovich and M.~I.~Vishik
\cite{AgranovichVishik64}, M.~S.~Agranovich \cite{Agranovich90, Agranovich92},
G.~Grubb \cite[Ch. 2]{Grubb96}, A.~N.~Kozhevnikov [7--10], R.~Denk, R.~Mennicken,
and L.~R.~Volevich \cite{DenkMennickenVolevich98, DenkMennickenVolevich01}, R.~Denk
and M.~Fairman \cite{DenkFairman10} and other papers (also see the surveys
\cite{Agranovich94, Agranovich97} and the references therein). Such classes have
important applications to the spectral theory of elliptic operators, to parabolic
differential equations and other; note that the most significant results are
obtained in the case of Hilbert spaces.

In this connection, of interest is the investigation of parameter--elliptic
operators in classes of Hilbert spaces that are calibrated much finer than the
Sobolev scale. For such classes, a sufficiently general function, not a number
parameter, serves as the smoothness index. Among them, we consider the class of all
Hilbert spaces that are interpolation spaces for the Hilbert Sobolev scale. This
class consists of the H\"ormander spaces $B_{2,k}$ \cite[Sec. 2.2]{Hermander63} for
which the smoothness index $k$ is an arbitrary radial function RO-vary\-ing at
$+\infty$. Such a class is naturally to call the extended Sobolev scale (by means of
inter\-po\-la\-tion spaces); this scale is distinguished and investigated in
\cite{09Dop3} and \cite[Sec. 2.4]{MikhailetsMurach10}. Since the isomorphism and
Fredholm properties of linear operators are preserved under the interpolation of
spaces, the extended Sobolev scale proved to be convenient and efficient in the
theory of general elliptic operators (see \cite{09Dop3, 09UMJ3, 12UMJ11} and
\cite[Sec 2.4.3]{MikhailetsMurach10}).

In this paper we investigate parameter--elliptic pseudodifferential operators given
on a closed smooth manifold and acting on the extended Sobolev scale. Our purpose is
to show that these operators possess the above--mentioned property on this scale.
Namely, we will prove a theorem on isomorphisms realized by a parameter--elliptic
pseudodifferential operator and on a priory estimates of a solution to the
corresponding elliptic equation.

Note that the theory of general elliptic equations and elliptic boundary--value
problems is built for a narrower class of H\"ormander spaces (called the refined
Sobolev scale) by V.~A.~Mikhailets and the second author in series of papers, among
them we mention the articles [21--28], survey \cite{12BJMA2}, and monograph
\cite{MikhailetsMurach10}. Specifically, parameter--elliptic equations are
investigated therein.

Nowadays H\"ormander spaces and their various analogs, called the spaces of
generalized smoothness, are of considerable interest both by themselves and to
applications [30--33].

\section{Statement of the problem}\label{sec2}

Let $\Gamma$ be a closed (i.e. compact and without boundary) infinitely smooth
manifold of dimension $n\geq1$. A certain $C^{\infty}$-density $dx$ is supposed to
be given on $\Gamma$. The linear topological spaces $C^{\infty}(\Gamma)$ of test
functions and $\mathcal{D}'(\Gamma)$ of distributions defined on $\Gamma$ are
considered as antidual spaces with respect to the inner product in
$L_{2}(\Gamma,dx)$. We suppose that functions and distributions are complex-valued,
and interpret distributions as antilinear functionals.

Following \cite[Sec. 4.1]{Agranovich94}, we recall the definition of a
parameter--elliptic pseudodifferential operator on $\Gamma$.

Let $\Psi_{\mathrm{ph}}^{r}(\Gamma)$ denote the class of polyhomogeneous (i.e.
classical) pseudodifferential operators (PsDOs) of order ${r}\in\mathbb{R}$ defined
on the mani\-fold~$\Gamma$. The principal symbol of a PsDO belonging to
$\Psi_{\mathrm{ph}}^{r}(\Gamma)$ is an infinitely smooth and complex-valued function
defined on the cotangent bundle $T^{\ast}\Gamma\setminus0$ (here $0$ is the
zero-section) and being positively homogeneous of the degree $r$ with respect to
$\xi$ in every section $T^{\ast}_{x}\Gamma\setminus\{0\}$, where $x\in\Gamma$. We
admit that the principal symbol can be equal to zero identically, then
$\Psi_{\mathrm{ph}}^{r}(\Gamma)\subset\Psi_{\mathrm{ph}}^{k}(\Gamma)$ whenever
$r<k$. A linear differential operator of order $r\geq1$ given on $\Gamma$ and having
infinitely smooth coefficients is an important special case of a PsDO belonging to
$\Psi_{\mathrm{ph}}^{r}(\Gamma)$. Note that the PsDOs under consideration are linear
and continuous on both topological spaces $C^{\infty}(\Gamma)$ and
$\mathcal{D}'(\Gamma)$.

Let numbers $m>0$ and $q\in\mathbb{N}$ be chosen arbitrarily. We consider a PsDO
$A(\lambda)$ that belongs to $\Psi_{\mathrm{\mathrm{ph}}}^{mq}(\Gamma)$ and depends
on the complex-valued parameter $\lambda$ in the following way:
\begin{equation}\label{f1}
A(\lambda)=\,\sum_{j=0}^{q}\,\lambda^{q-j}\,A_{j}.
\end{equation}
Here $A_{j}\in\Psi_{\mathrm{\mathrm{ph}}}^{mj}(\Gamma)$ for each
$j\in\{0,\ldots,q\}$, and moreover $A_{0}$ is an operator of multiplication by a
function $a_{0}\in C^{\infty}(\Gamma)$. Note that since
$m(q-j)+\mathrm{ord}\,A_{j}=\mathrm{ord}\,A(\lambda)$, the weight $m$ is ascribed to
$\lambda$ in \eqref{f1}.

Let $K$ be a fixed closed angle on the complex plain with the vertex at the origin
(we do not exclude the case where $K$ degenerates into a ray).

The PsDO $A(\lambda)$ is said to be parameter--elliptic in the angle $K$ on the
manifold $\Gamma$ if
\begin{equation}\label{f2}
\sum_{j\,=\,0}^{q}\,\lambda^{q-j}\,a_{j,0}(x,\xi)\neq0
\end{equation}
for each point $x\in\Gamma$, covector $\xi\in T_{x}^{\ast}\Gamma$ and parameter
$\lambda\in K$ such that $(\xi,\lambda)\neq0$. Here $a_{j,0}(x,\xi)$ is the
principle symbol of $A_{j}$, so $a_{0,0}(x,\xi)\equiv a_{0}(x)$. We also admit that
the functions $a_{1,0}(x,\xi),\,a_{2,0}(x,\xi),\ldots$ are equal to zero at $\xi=0$
(this assumption is connected with the fact that the principal symbols are not
initially defined at $\xi=0$).

For instance, let a PsDO be of the form $A-\lambda I$, where
$A\in\Psi_{\mathrm{\mathrm{ph}}}^{m}(\Gamma)$ (as usual $I$ denotes the identical
operator). Then, for $A-\lambda I$, the parameter-ellipticity condition in $K$ means
that $a_{0}(x,\xi)\notin K$ whenever $\xi\neq0$; here $a_{0}(x,\xi)$ is  the
principal symbol of $A$. This example is important in the spectral theory of
elliptic operators.

We investigate properties of the parameter--elliptic PsDO $A(\lambda)$ on the
extended Sobo\-lev scale.

\section{The extended Sobolev scale}\label{sec3}

Following \cite[Sec. 2.4]{MikhailetsMurach10}, we will introduce the spaces that
form the extended Sobolev scale. They are parametrized with a function
$\varphi\in\mathrm{RO}$, which characterizes regularity properties of the
distributions belonging to the space. Here RO is the set of all Borel measurable
functions $\varphi:[1,\infty)\rightarrow(0,\infty)$ for which there exist numbers
$a>1$ and $c\geq1$ such that
\begin{equation}\label{f3}
c^{-1}\leq\frac{\varphi(\lambda t)}{\varphi(t)}\leq c\quad\mbox{for each}\quad
t\geq1\quad\mbox{and}\quad\lambda\in[1,a]
\end{equation}
(generally, the constants $a$ and $c$ depend on $\varphi\in\mathrm{RO}$). These
functions are said to be RO-varying at $+\infty$. The class of RO-varying functions
was introduced by V.~G.~Avakumovi\'c \cite{Avakumovic36} in 1936 and has been
sufficiently investigated \cite{Seneta76, BinghamGoldieTeugels89}.

The class RO is admitted the following description
\begin{equation*}
\varphi\in\mathrm{RO}\quad\Longleftrightarrow\quad\varphi(t)=
\exp\biggl(\beta(t)+\int_{1}^{t}\frac{\gamma(\tau)}{\tau}\;d\tau\biggr), \;\;t\geq1,
\end{equation*}
where the real-valued functions $\beta$ and $\gamma$ are Borel measurable and
bounded on $[1,\infty)$. Note also that condition \eqref{f3} is equivalent to the
bilateral inequality
\begin{equation}\label{f4}
c^{-1}\lambda^{s_{0}}\leq\frac{\varphi(\lambda t)}{\varphi(t)}\leq
c\lambda^{s_{1}}\quad\mbox{for each}\quad t\geq1\quad\mbox{and}\quad\lambda\geq1,
\end{equation}
in which (another) constant $c\geq1$ is independent of $t$ and $\lambda$. Hence, for
every function $\varphi\in\mathrm{RO}$, we may define the lower and the upper
Matuszewska indices \cite{Matuszewska64} as follows:
\begin{gather}\label{eq6}
\sigma_{0}(\varphi):=\sup\{s_{0}\in\mathbb{R}:\,\mbox{the left-hand inequality in
\eqref{f4} holds}\},\\ \label{eq7}
\sigma_{1}(\varphi):=\inf\{s_{1}\in\mathbb{R}:\,\mbox{the right-hand inequality in
\eqref{f4} holds}\}
\end{gather}
(see \cite[Theorem 2.2.2]{BinghamGoldieTeugels89}); here
$-\infty<\sigma_{0}(\varphi)\leq\sigma_{1}(\varphi)<\infty$.

Now let $\varphi\in\mathrm{RO}$ and introduce the necessary function spaces over
$\mathbb{R}^{n}$ and then over~$\Gamma$.

The linear space $H^{\varphi}(\mathbb{R}^{n})$ is defined to consists of all
distributions $w\in\mathcal{S}'(\mathbb{R}^{n})$ such that their Fourier transform
$\widehat{w}:=\mathcal{F}w$ is locally Lebesgue integrable over $\mathbb{R}^{n}$ and
satisfies the condition
$$
\int_{\mathbb{R}^{n}}\varphi^2(\langle\xi\rangle)\,|\widehat{w}(\xi)|^2\,d\xi
<\infty.
$$
Here, as usual, $\mathcal{S}'(\mathbb{R}^{n})$ is the linear topological space of
tempered distributions given in $\mathbb{R}^{n}$, and
$\langle\xi\rangle:=(1+|\xi|^{2})^{1/2}$ is the smoothed modulus of
$\xi\in\mathbb{R}^{n}$. The inner product in $H^{\varphi}(\mathbb{R}^{n})$ is
defined by the formula
$$
(w_1,w_2)_{H^{\varphi}(\mathbb{R}^{n})}:=
\int_{\mathbb{R}^{n}}\varphi^2(\langle\xi\rangle)\,
\widehat{w_1}(\xi)\,\overline{\widehat{w_2}(\xi)}\,d\xi.
$$
It endows $H^{\varphi}(\mathbb{R}^{n})$ with the Hilbert space structure and induces
the norm
$\|w\|_{H^{\varphi}(\mathbb{R}^{n})}:=(w,w)_{H^{\varphi}(\mathbb{R}^{n})}^{1/2}$.

The space $H^{\varphi}(\mathbb{R}^{n})$ is a Hilbert and isotropic case of the
spaces $B_{p,k}$ introduced and systematically investigated by L.~H\"ormander
\cite[Sec. 2.2]{Hermander63} (also see \cite[Sec. 10.1]{Hermander83}). Namely,
$H^{\varphi}(\mathbb{R}^{n})=B_{p,k}$ provided $p=2$ and
$k(\xi)=\varphi(\langle\xi\rangle)$ for all $\xi\in\mathbb{R}^{n}$. Not that, if
$p=2$, then the H\"ormander spaces coincide with the spaces introduced and
investigated by L.~R.~Volevich and B.~P.~Paneah \cite[Sec.~2]{VolevichPaneah65}.

To define an analog of $H^{\varphi}(\mathbb{R}^{n})$ for the manifold $\Gamma$,
choose a finite atlas belonging to the $C^{\infty}$-structure on $\Gamma$. Let this
atlas consist of local charts $\alpha_{j}:\mathbb{R}^{n}\leftrightarrow\Gamma_{j}$,
$j=1,\ldots,p$, where the open sets  $\Gamma_{j}$ form a finite covering of
$\Gamma$. Also choose functions $\chi_{j}\in C^{\infty}(\Gamma)$, $j=1,\ldots,p$,
that satisfy the condition $\mathrm{supp}\,\chi_{j}\subset\Gamma_{j}$ and that form
a partition of unity on $\Gamma$.

The linear space $H^{\varphi}(\Gamma)$ is defined to consist of all distributions
$u\in\nobreak\mathcal{D}'(\Gamma)$ such that $(\chi_{j}u)\circ\alpha_{j}\in
H^{\varphi}(\mathbb{R}^{n})$ for every $j\in\{1,\ldots,p\}$. Here
$(\chi_{j}u)\circ\alpha_{j}$ is the representation of the distribution $\chi_{j}u$
in the local chart $\alpha_{j}$. The inner product in $H^{\varphi}(\Gamma)$ is
defined by the formula
$$
(u_{1},u_{2})_{\varphi}:=\sum_{j=1}^{p}\,((\chi_{j}u_{1})\circ\alpha_{j},
(\chi_{j}u_{2})\circ\alpha_{j})_{H^{\varphi}(\mathbb{R}^{n})},
$$
where $u_{1},u_{2}\in H^{\varphi}(\mathbb{R}^{n})$. This inner product endows
$H^{\varphi}(\mathbb{R}^{n})$ with the Hilbert space structure and induces the norm
$\|u\|_{\varphi}:=(u,u)_{\varphi}^{1/2}$.

The Hilbert space $H^{\varphi}(\Gamma)$ does not depend (up to equivalence of norms)
on our choice of local charts and partition of unity on~$\Gamma$ \cite[Theorem
2.21]{MikhailetsMurach10}. This space is separable, and the continuous and dense
embeddings $C^{\infty}(\Gamma)\hookrightarrow
H^{\varphi}(\Gamma)\hookrightarrow\mathcal{D}'(\Gamma)$ hold.

If $\varphi(t)=t^{s}$ for each $t\geq1$ with some $s\in\mathbb{R}$, then
$H^{\varphi}(\mathbb{R}^{n})=:H^{(s)}(\mathbb{R}^{n})$ and
$H^{\varphi}(\Gamma)=:H^{(s)}(\Gamma)$ are the inner product Sobolev spaces (of the
differentiation order $s$) given over $\mathbb{R}^{n}$ and $\Gamma$ respectively.

The class of Hilbert function spaces
$$
\{H^{\varphi}(\mathbb{R}^{n}\;\mbox{or}\;\Gamma):\varphi\in\mathrm{RO}\}
$$
is naturally said to be the extended Sobolev scale over $\mathbb{R}^{n}$ or
$\Gamma$.

We mention some properties of the extended Sobolev scale on $\Gamma$ connected with
embedding of spaces. Let $\varphi,\varphi_{1}\in\mathrm{RO}$; the function
$\varphi(t)/\varphi_{1}(t)$ is bounded on a neighbourhood of $+\infty$ if and only
if $H^{\varphi_{1}}(\Gamma)\hookrightarrow H^{\varphi}(\Gamma)$. This embedding is
continuous and dense; moreover, it is compact if and only if
$\varphi(t)/\varphi_{1}(t)\rightarrow0$ as $t\rightarrow+\infty$. Specifically, the
following compact and dense embeddings hold:
\begin{equation}\label{f7}
H^{(s_1)}(\Gamma)\hookrightarrow H^{\varphi}(\Gamma)\hookrightarrow
H^{(s_0)}(\Gamma)\quad\mbox{for each}\quad
s_{1}>\sigma_{1}(\varphi)\quad\mbox{and}\quad s_{0}<\sigma_{0}(\varphi).
\end{equation}
This properties result from the corresponding properties of the H\"ormander spaces
$B_{2,k}$ \cite[Sec. 2.2]{Hermander63}.

\section{The main result}\label{sec4}

Put $\varrho(t):=t$ for $t\geq1$. The PsDO $A(\lambda)$, which order is $mq$,
defines the bounded operator
\begin{equation}\label{f8}
A(\lambda):\,\,H^{\varphi\varrho^{mq}}(\Gamma)\,\rightarrow\,H^{\varphi}(\Gamma)
\quad\mbox{for each}\quad\lambda\in\mathbb{C}\quad\mbox{and}
\quad\varphi\in\mathrm{RO}.
\end{equation}
This fact will be proved in Section \ref{sec6}. Note here that
$\varphi\rho^{mq}\in\mathrm{RO}$, and therefore operator \eqref{f8} acts on the
extended Sobolev scale.

The main result of the paper is the following.

\medskip

\noindent\textbf{Theorem.} \it Suppose that the PsDO $A(\lambda)$ is
parameter--elliptic in the corner $K\subset\mathbb{C}$ on the manifold $\Gamma$.
Then there exists a number $\lambda_{0}>0$ such that for every $\lambda\in K$ and
$\varphi\in\mathrm{RO}$ we have the isomorphism
\begin{equation}\label{f9}
A(\lambda):\,\,H^{\varphi\varrho^{mq}}(\Gamma)\,\leftrightarrow\,H^{\varphi}(\Gamma)
\quad\mbox{whenever}\quad|\lambda|\geq\lambda_{0}.
\end{equation}
Moreover, for each fixed $\varphi\in\mathrm{RO}$ there exists a number
$c=c(\varphi)\geq1$ such that
\begin{equation}\label{f10}
c^{-1}\,\|A(\lambda)u\|_{\varphi}\leq\bigl(\,\|u\|_{\varphi\varrho^{mq}}+
|\lambda|^{q}\,\|\,u\|_{\varphi}\,\bigr)\leq c\,\|A(\lambda)u\|_{\varphi}
\end{equation}
for every $\lambda\in K$, with $|\lambda|\geq\lambda_{0}$, and all $u\in
H^{s+mq,\varphi}(\Gamma)$. Here the number $c$ does not depend on $\lambda$ and $u$.
\rm

\medskip

This theorem is known in the Sobolev case, where $\varphi(t)\equiv t^{s}$ and
$s\in\mathbb{R}$ (see \cite[Theorem 4.1.2]{Agranovich94}). We will prove Theorem for
arbitrary $\varphi\in\mathrm{RO}$ in Section~\ref{sec7} by applying interpolation
with function parameter.

Note that the left-hand inequality in \eqref{f10} remains true without the
parameter--ellip\-ti\-city assumption (see Lemma~\ref{lem2} in Section~\ref{sec6}).

\section{Interpolation with function parameter}\label{sec5}

The extended Sobolev scale possesses an important interpolation property, which we
will use. Namely, every space $H^{\varphi}(\Gamma)$, with $\varphi\in\mathrm{RO}$,
is the result of the interpolation (with an appropriate function parameter) between
the Sobolev spaces $H^{(s_0)}(\Gamma)$ and $H^{(s_1)}(\Gamma)$ appearing in
\eqref{f7}. (An analogous result holds for the spaces defined over
$\mathbb{R}^{n}$.) In this connection we recall the definition of interpolation with
function parameter in the case of general Hilbert spaces and then state some
properties of the interpolation (see \cite[Sec. 1.1]{MikhailetsMurach10}). It is
sufficient to restrict ourselves to separable complex Hilbert spaces.

Let $X:=[X_{0},X_{1}]$ be an ordered couple of separable complex Hilbert spaces such
that the continuous and dense embedding $X_{1}\hookrightarrow X_{0}$ holds. We say
that this couple is admissible. For $X$ there exists an isometric isomorphism
$J:X_{1}\leftrightarrow X_{0}$ such that $J$ is a self-adjoint positive operator on
$X_{0}$ with the domain $X_{1}$. The operator $J$ is called a generating operator
for the couple~$X$. This operator is uniquely determined by $X$.

Let $\psi\in\mathcal{B}$, where $\mathcal{B}$ is the set of all Borel measurable
functions $\psi:(0,\infty)\rightarrow(0,\infty)$ such that $\psi$ is bounded on each
compact interval $[a,b]$, with $0<a<b<\infty$, and that $1/\psi$ is bounded on every
set $[r,\infty)$, with $r>0$.

Consider the operator $\psi(J)$, which is defined (and positive) in $X_{0}$ as the
Borel function $\psi$ of $J$. Denote by $[X_{0},X_{1}]_{\psi}$ or simply by
$X_{\psi}$ the domain of the operator $\psi(J)$ endowed with the inner product
$(u_{1},u_{2})_{X_{\psi}}:=(\psi(J)u_{1},\psi(J)u_{2})_{X_{0}}$ and the
corresponding norm $\|u\|_{X_{\psi}}=\|\psi(J)u\|_{X_{0}}$. The space $X_{\psi}$ is
Hilbert and separable.

A function $\psi\in\mathcal{B}$ is called an interpolation parameter if the
following condition is fulfilled for all admissible couples $X=[X_{0},X_{1}]$ and
$Y=[Y_{0},Y_{1}]$ of Hilbert spaces and for an arbitrary linear mapping $T$ given on
$X_{0}$: if the restriction of $T$ to $X_{j}$ is a bounded operator
$T:X_{j}\rightarrow Y_{j}$ for each $j\in\{0,1\}$, then the restriction of $T$ to
$X_{\psi}$ is also a bounded operator $T:X_{\psi}\rightarrow Y_{\psi}$.

If $\psi$ is an interpolation parameter, then we say that the Hilbert space
$X_{\psi}$ is obtained by the interpolation of $X$ with the function parameter
$\psi$. In this case, the dense and continuous embeddings $X_{1}\hookrightarrow
X_{\psi}\hookrightarrow X_{0}$ hold.

Note that a function $\psi\in\mathcal{B}$ is an interpolation parameter if and only
if $\psi$ is pseudoconcave on a neighborhood of $+\infty$ (see \cite[Theorem
1.9]{MikhailetsMurach10}). The latter condition means that there exists a concave
function $\psi_{1}:(b,\infty)\rightarrow(0,\infty)$, with $b\gg1$, such that both
functions $\psi/\psi_{1}$ and $\psi_{1}/\psi$ are bounded on $(b,\infty)$.

The above-mentioned interpolation property of the extended Sobolev is stated in the
following way \cite[Theorems 2.18 and 2.22]{MikhailetsMurach10}.

\begin{proposition}\label{prop1}
Let a function $\varphi\in\mathrm{RO}$ and numbers $s_{0},s_{1}\in\mathbb{R}$ be
such that $s_{0}<\sigma_{0}(\varphi)$ and $s_{1}>\sigma_{1}(\varphi)$. Set
\begin{equation}\label{f11}
\psi(t):=
\begin{cases}
\;t^{{-s_0}/{(s_1-s_0)}}\,
\varphi\bigl(t^{1/{(s_1-s_0)}}\bigr)&\text{for}\quad t\geq1, \\
\;\varphi(1)&\text{for}\quad0<t<1.
\end{cases}
\end{equation}
Then $\psi\in\mathcal{B}$ is an interpolation parameter, and
\begin{gather*}
[H^{(s_0)}(\mathbb{R}^{n}),H^{(s_1)}(\mathbb{R}^{n})]_{\psi}=H^{\varphi}(\mathbb{R}^{n})
\quad \mbox{with equality of norms},\\
[H^{(s_0)}(\Gamma),H^{(s_1)}(\Gamma)]_{\psi}=H^{\varphi}(\Gamma)\quad \mbox{with
equivalence of norms}.
\end{gather*}
\end{proposition}

We will also use two properties of interpolation between abstract Hilbert spaces.
The first of them is the following estimate of the operator norm in interpolation
spaces \cite[Theorem 1.8]{MikhailetsMurach10}.

\begin{proposition}\label{prop2}
For every interpolation parameter $\psi\in\mathcal{B}$ there exists a number
$\widetilde{c}=\widetilde{c}(\psi)>0$ such that
$$
\|T\|_{X_{\psi}\rightarrow Y_{\psi}}\leq
\widetilde{c}\,\max\,\bigl\{\,\|T\|_{X_{j}\rightarrow Y_{j}}:\,j=0,\,1\,\bigr\}.
$$
Here $X=[X_{0},X_{1}]$ and $Y=[Y_{0},Y_{1}]$ are arbitrary normal admissible couples
of Hilbert spaces, and $T$ is an arbitrary linear mapping given on $X_{0}$ and
defining the bounded operators $T:X_{j}\rightarrow Y_{j}$, with  $j=0,\,1$. The
number $c_{\psi}>0$ does not depend on $X$, $Y$, and $T$.
\end{proposition}

Recall here that an admissible couple of Hilbert spaces $X=[X_{0},X_{1}]$ is said to
be normal if $\|u\|_{X_{0}}\leq\|u\|_{X_{1}}$ for each $u\in X_{1}$. Note that each
admissible couple $[X_{0},X_{1}]$ can be transformed into a normal couple by
replacing the norm $\|\,u\,\|_{X_{0}}$ with the proportional norm
$k\,\|u\|_{X_{0}}$, where $k$ is the norm of the embedding operator
$X_{1}\hookrightarrow X_{0}$.

The second property is useful when we interpolate between direct sums of Hilbert
spaces.

\begin{proposition}\label{prop3}
Let $\bigl[X_{0}^{(j)},X_{1}^{(j)}\bigr]$, with $j=1,\ldots,p$, be a finite
collection of admissible couples of Hilbert spaces. Then for every function
$\psi\in\mathcal{B}$ we have
$$
\biggl[\,\bigoplus_{j=1}^{p}X_{0}^{(j)},\,\bigoplus_{j=1}^{p}X_{1}^{(j)}\biggr]_{\psi}=\,
\bigoplus_{j=1}^{p}\bigl[X_{0}^{(j)},\,X_{1}^{(j)}\bigr]_{\psi}\quad\mbox{with
equality of norms}.
$$
\end{proposition}

\section{Some auxiliary results}\label{sec6}

Here we will prove some auxiliary results regarding the boundedness of the PsDO
$A(\lambda)$ on the extended Sobolev scale.

\begin{lemma}\label{lem1}
Let $T\in\Psi_{\mathrm{ph}}^{r}(\Gamma)$ for some $r\in\mathbb{R}$. Then the PsDO T
defines the bounded operator
$$
T:H^{\varphi\varrho^{r}}(\Gamma)\rightarrow H^{\varphi}(\Gamma)\quad\mbox{for
each}\quad\varphi\in\mathrm{RO}.
$$
\end{lemma}

\begin{proof}
This lemma is known in the Sobolev case \cite[Theorem 2.1.2]{Agranovich94}. We prove
the lemma for arbitrary $\varphi\in\mathrm{RO}$ by applying Proposition~\ref{prop1}.
Choose numbers $s_{0}$ and $s_{1}$ so that $s_{0}<\sigma_{0}(\varphi)$ and
$s_{1}>\sigma_{1}(\varphi)$. Let $\psi$ be the interpolation parameter appearing in
Proposition~\ref{prop1}. Consider the bounded operators
\begin{equation}\label{f13}
T:\,H^{(s_{j}+r)}(\Gamma)\rightarrow H^{(s_{j})}(\Gamma),\quad j=0,1,
\end{equation}
which map between Sobolev spaces. Applying the interpolation with the function
para\-meter $\psi$ to \eqref{f13}, we obtain, by Proposition~\ref{prop1}, the
bounded operator required
\begin{gather*}
T:\,H^{\varphi\varrho^{r}}(\Gamma)=
\bigl[H^{(s_0+r)}(\Gamma),H^{(s_1+r)}(\Gamma)\bigr]_\psi
\rightarrow\bigl[H^{(s_0)}(\Gamma),H^{(s_1)}(\Gamma)\bigr]_\psi=H^{\varphi}(\Gamma).
\end{gather*}
Note that the first equality is true here because
$s_{0}+r<\sigma_{0}(\varphi\rho^{r})$, $s_{1}+r>\sigma_{1}(\varphi\rho^{r})$, and
$\psi$ satisfies formula \eqref{f11}, in which $s_0$, $s_1$, and $\varphi$ should be
replaced with $s_0+r$, $s_1+r$, and $\varphi\rho^{r}$ respectively.
\end{proof}

According to Lemma~\ref{lem1}, the operator \eqref{f8} is well-defined and bounded
for each $\lambda\in\mathbb{C}$ and $\varphi\in\mathrm{RO}$. The next lemma refines
this result.

\begin{lemma}\label{lem2}
For an arbitrary $\varphi\in\mathrm{RO}$ there exists a number $c'=c'(\varphi)>0$
such that
\begin{equation}\label{f14}
\|A(\lambda)u\|_{\varphi}\leq c'\bigl(\,\|u\|_{\varphi\varrho^{mq}}+
|\lambda|^{q}\,\|\,u\|_{\varphi}\,\bigr)
\end{equation}
for every $\lambda\in\mathbb{C}$ and each $u\in H^{\varphi\varrho^{mq}}(\Gamma)$.
Here $c'$ does not depend on $\lambda$ and~$u$.
\end{lemma}

\begin{proof}
We will use the following interpolation inequality:
\begin{equation}\label{f15}
r^{\varepsilon}\,\|u\|_{\eta}\leq \sqrt{2}\,\bigl(\,\|u\|_{\eta\varrho^\varepsilon}+
r^{\varepsilon+\delta}\,\|u\|_{\eta\varrho^{-\delta}}\bigr),
\end{equation}
where the number parameters $r,\varepsilon,\delta\geq0$, function parameter
$\eta\in\mathrm{RO}$ and distribution $u\in H^{\eta\varrho^\varepsilon}(\Gamma)$ are
all arbitrary.  (Similar inequalities are known for Sobolev spaces; see, e.g.,
\cite[\S~1, Sec.~6]{AgranovichVishik64}.)

Formula \eqref{f15} follows from the evident inequality
$1\leq(k/r)^{\varepsilon}+(r/k)^{\delta}$ for all positive numbers $r$ and $k$.
Indeed, if we put $k:=\langle\xi\rangle$ in this inequality and multiply its sides
by $r^{\varepsilon}\eta(\langle\xi\rangle)\,|\widehat{w}(\xi)|$, where
$\xi\in\mathbb{R}^{n}$ and $w\in H^{\eta\varrho^\varepsilon}(\mathbb{R}^{n})$ are
arbitrary, then we obtain an analog of \eqref{f15} for spaces over $\mathbb{R}^{n}$.
Namely, we may write the following:
\begin{gather*}
r^{\varepsilon}\,\|w\|_{H^{\eta}(\mathbb{R}^{n})}=
\|\,r^{\varepsilon}\eta(\langle\xi\rangle)\,
|\widehat{w}(\xi)|\,\|_{L_{2}(\mathbb{R}^{n},d\xi)}\\
\leq\|\,\eta(\langle\xi\rangle)\langle\xi\rangle^{\varepsilon}
|\widehat{w}(\xi)|\,\|_{L_{2}(\mathbb{R}^{n},d\xi)}+
\|\,r^{\varepsilon+\delta}\eta(\langle\xi\rangle)\langle\xi\rangle^{-\delta}
|\widehat{w}(\xi)|\,\|_{L_{2}(\mathbb{R}^{n},d\xi)}\\
=\|w\|_{H^{\eta\varrho^\varepsilon}(\mathbb{R}^{n})}+
r^{\varepsilon+\delta}\,\|w\|_{H^{\eta\varrho^{-\delta}}(\mathbb{R}^{n})}.
\end{gather*}
Here, as usual, $L_{2}(\mathbb{R}^{n},d\xi)$ denotes the Hilbert space of functions
square integrable over $\mathbb{R}^{n}$ with respect to the Lebesgue measure $d\xi$,
where $\xi$ is their argument. Whence we directly obtain \eqref{f15} according to
the definition of the spaces over $\Gamma$. Certainly, we should use the same
collection of local charts and partition of unity for these spaces.

Now let $\varphi\in\mathrm{RO}$ be chosen arbitrarily. Then for each
$\lambda\in\mathbb{C}$ and $u\in H^{\varphi\varrho^{mq}}(\Gamma)$, we may write
\begin{gather*}
\|A(\lambda)u\|_{\varphi}\leq\sum_{j=0}^{q}\,|\lambda|^{q-j}\,\|A_{j}u\|_{\varphi}\leq
c_{1}\sum_{j=0}^{q}\,|\lambda|^{q-j}\,\|u\|_{\varphi\varrho^{mj}}\\
\leq c_{1}\,\sqrt{2}\,\bigl(\,\|u\|_{\varphi\varrho^{mq}}+
|\lambda|^{q}\,\|\,u\|_{\varphi}\,\bigr).
\end{gather*}
Here we apply \eqref{f1}, Lemma~\ref{lem1}, and \eqref{f15} in succession. According
to Lemma~\ref{lem1}, the number $c_{1}>0$ is independent of both $\lambda$ and $u$
in these inequalities. Note that we use \eqref{f15} for $\eta:=\varphi\varrho^{mj}$,
$\varepsilon:=m(q-j)$, $\delta:=mj$ and $r:=|\lambda|^{1/m}$, with $j=0,\ldots,q$.
Thus, we have the required inequality \eqref{f14} with $c':=c_{1}\,\sqrt{2}$.
\end{proof}

\section{Proof of the main result}\label{sec7}

Our proof of Theorem is based on an interpolation property of some
parameter--dependent spaces. Therefore we will first introduce these spaces,
establish this property, and then prove Theorem.

Let a function $\eta\in\mathrm{RO}$ and numbers $r,\theta\geq0$ be given. We let
$H^{\eta}(\Gamma,r,\theta)$ denote the space $H^{\eta}(\Gamma)$ which is endowed
with the norm depending on the parameters $r$ and $\theta$ in the following way
$$
\|u\|_{\eta,r,\theta}:=\bigl(\;\|u\|_{\eta}^{2}+r^{2}\,
\|u\|_{\eta\varrho^{-\theta}}^{2}\bigr)^{1/2},\quad u\in H^{\eta}(\Gamma).
$$
The space $H^{\eta}(\Gamma,r,\theta)$ is well-defined, and the norms in
$H^{\eta}(\Gamma,r,\theta)$ and $H^{\eta}(\Gamma)$ are equivalent. This directly
follows from the continuous embedding $H^{\eta}(\Gamma)\hookrightarrow
H^{\eta\varrho^{-\theta}}(\Gamma)$.  Note that the norm in the space
$H^{\eta}(\Gamma,r,\theta)$ is induced by the inner product
$$
(u_{1},u_{2})_{\eta,r,\theta}:=(u_{1},u_{2})_{\eta}+
r^{2}\,(u_{1},u_{2})_{\eta\varrho^{-\theta}},\quad u_{1},u_{2}\in H^{\eta}(\Gamma);
$$
therefore this space is Hilbert. If we consider the Sobolev case where
$\eta(t)\equiv t^{s}$ for some $s\in\mathbb{R}$, then the space
$H^{\eta}(\Gamma,r,\theta)$ are denoted by $H^{(s)}(\Gamma,r,\theta)$.

Returning to Theorem, note that
\begin{equation}\label{f7.14}
\|u\|_{\varphi\varrho^{mq},|\lambda|^{q},mq}\leq
(\,\|u\|_{\varphi\varrho^{mq}}+|\lambda|^{q}\,\|\,u\|_{\varphi}\,\bigr)\leq
\sqrt{2}\;\|u\|_{\varphi\varrho^{mq},|\lambda|^{q},mq}
\end{equation}
for each $u\in H^{\varphi\varrho^{mq}}(\Gamma)$.

According to Proposition~\ref{prop1}, the spaces
$$
\bigl[\,H^{(l_0)}(\Gamma,r,\theta),\,H^{(l_1)}(\Gamma,r,\theta)\,\bigr]_{\psi}
\quad\mbox{and}\quad H^{\eta}(\Gamma,r,\theta)
$$
are equal up to equivalence of norms. Here both the numbers $l_{0}<\sigma_{0}(\eta)$
and $l_{1}>\sigma_{1}(\eta)$ are arbitrary, whereas the interpolation parameter
$\psi$ is defined by the formula
\begin{equation}\label{f7.15}
\psi(t):=
\begin{cases}
\;t^{{-l_0}/{(l_1-l_0)}}\,
\eta\bigl(t^{1/{(l_1-l_0)}}\bigr)&\text{for}\quad t\geq1, \\
\;\eta(1)&\text{for}\quad0<t<1.
\end{cases}
\end{equation}
We now refine this result in the following way.

\begin{lemma}\label{lem3}
Let a function $\eta\in\mathrm{RO}$ and numbers $l_{0}<\sigma_{0}(\eta)$,
$l_{1}>\sigma_{1}(\eta)$, $\theta\geq0$ be all chosen arbitrarily. Then there exist
a number $c_{0}\geq1$ such that
\begin{equation}\label{f7.16}
c_{0}^{-1}\,\|u\|_{\eta,r,\theta}\leq
\|u\|_{[H^{(l_0)}(\Gamma,r,\theta),H^{(l_1)}(\Gamma,r,\theta)]_{\psi}}\leq
c_{0}\,\|u\|_{\eta,r,\theta}
\end{equation}
for every number $r\geq0$ and each distribution $u\in H^{\eta}(\Gamma)$. Here $\psi$
is the interpolation parameter defined by \eqref{f7.15}, and the number $c_{0}$ does
not depend on $r$ and $u$.
\end{lemma}

\begin{proof}
Let a number $r\geq0$ be arbitrary. We will first prove the following: if we replace
$\Gamma$ with $\mathbb{R}^{n}$ in formula \eqref{f7.16}, then it holds for
$c_{0}=1$.

Let $H^{\eta}(\mathbb{R}^{n},r,\theta)$ denote the space $H^{\eta}(\mathbb{R}^{n})$
endowed with the Hilbert norm
\begin{gather}\notag
\|w\|_{H^{\eta}(\mathbb{R}^{n},r,\theta)}:=
\Bigl(\;\|w\|_{H^{\eta}(\mathbb{R}^{n})}^{2}+ r^{2}\,\|w\|
_{H^{\eta\varrho^{-\theta}}(\mathbb{R}^{n})}^{2}\:\Bigr)^{1/2}\\=
\biggl(\;\;\int_{\mathbb{R}^{n}}\,
\bigl(1+r^{2}\,\langle\xi\rangle^{-2\theta}\bigr)\,
\eta^{2}(\langle\xi\rangle)\,|\widehat{w}(\xi)|^{2}\,d\xi\,\biggr)^{1/2};\label{f7.17}
\end{gather}
here $w\in H^{\eta}(\mathbb{R}^{n})$. This norm is equivalent to the norm in
$H^{\eta}(\mathbb{R}^{n})$ for every fixed $r\geq0$. Hence, the space
$H^{\eta}(\mathbb{R}^{n},r,\theta)$ is Hilbert. If $\eta(t)\equiv t^{s}$ for some
$s\in\mathbb{R}$ (the Sobolev case), then the space
$H^{\eta}(\mathbb{R}^{n},r,\theta)$ is denoted by
$H^{(s)}(\mathbb{R}^{n},r,\theta)$.

Calculate the norm in the interpolation space
\begin{equation}\label{f7.18}
\bigl[\,H^{(l_{0})}(\mathbb{R}^{n},r,\theta),\,H^{(l_{1})}(\mathbb{R}^{n},r,\theta)
\,\bigr]_{\psi}.
\end{equation}
Let $J$ denote the PsDO in $\mathbb{R}^{n}$ with the symbol
$\langle\xi\rangle^{l_{1}-l_{0}}$, where $\xi\in\mathbb{R}^{n}$ is argument. We may
verify directly that $J$ is the generating operator for the couple of spaces
appearing in \eqref{f7.18}. Applying the isometric isomorphism
$$
\mathcal{F}:H^{(l_{0})}(\mathbb{R}^{n},r,\theta)\leftrightarrow
L_{2}\bigl(\mathbb{R}^{n}, (1+r^{2}\,\langle\xi\rangle^{-2\theta})\,
\langle\xi\rangle^{2l_{0}}\,d\xi\bigr),
$$
we reduce the operator $J$ to the form of multiplication by the function
$\langle\xi\rangle^{l_{1}-l_{0}}$; here $\mathcal{F}$ is the Fourier transform.
Therefore the operator $\psi(J)$ is reduced to the form of multiplication by the
function $\psi(\langle\xi\rangle^{l_{1}-l_{0}})=
\langle\xi\rangle^{-l_{0}}\eta(\langle\xi\rangle)$ in view of \eqref{f7.15}. Hence,
given any $w\in H^{\eta}(\mathbb{R}^{n})$, we have
\begin{gather*}
\|w\|^{2}_{[H^{(l_{0})}(\mathbb{R}^{n},r,\theta),
H^{(l_{1})}(\mathbb{R}^{n},r,\theta)]_{\psi}}=
\|\psi(J)\,w\|_{H^{(l_{0})}(\mathbb{R}^{n},r,\theta)}^{2}\\
=\int_{\mathbb{R}^{n}}\, \bigl|\langle\xi\rangle^{-l_{0}}\,\eta(\langle\xi\rangle)\,
\widehat{w}(\xi)\bigr|^{2}\,\bigl(1+r^{2}\,\langle\xi\rangle^{-2\theta}\bigr)\,
\langle\xi\rangle^{2l_{0}}\,d\xi
=\|w\|_{H^{\eta}(\mathbb{R}^{n},r,\theta)}^{2}<\infty;
\end{gather*}
here \eqref{f7.17} is used. Thus
\begin{equation}\label{f7.19}
\bigl[\,H^{(l_{0})}(\mathbb{R}^{n},r,\theta),\,H^{(l_{1})}(\mathbb{R}^{n},r,\theta)
\,\bigr]_{\psi}=H^{\eta}(\mathbb{R}^{n},r,\theta)\quad\mbox{with equality of norms}.
\end{equation}

We will now prove \eqref{f7.16} by applying property \eqref{f7.19} and the
definition of spaces over~$\Gamma$. Fix a finite atlas $\{\alpha_{j}\}$ and
partition of unity $\{\chi_{j}\}$ on $\Gamma$ used in this definition (see
Sec\-tion~\ref{sec3}); here $j=1,\ldots,p$.

Consider the linear mapping of the "rectification" of $\Gamma$, namely,
$$
T:\,u\mapsto\bigl((\chi_{1}u)\circ\alpha_{1},\ldots,
(\chi_{p}u)\circ\alpha_{p}\bigr),\quad u\in\mathcal{D}'(\Gamma).
$$
We may directly verify that this mapping defines the isometric operators
\begin{gather}\label{f7.20}
T:\,H^{\eta}(\Gamma,r,\theta)\rightarrow
\bigl(H^{\eta}(\mathbb{R}^{n},r,\theta)\bigr)^{p}, \\
T:\,H^{(l_{j})}(\Gamma,r,\theta)\rightarrow
\bigl(H^{(l_{j})}(\mathbb{R}^{n},r,\theta)\bigr)^{p}, \quad j\in\{0,\,1\}.
\label{f7.21}
\end{gather}
Applying the interpolation with the parameter $\psi$ to \eqref{f7.21}, we obtain the
bounded operator
\begin{equation}\label{f7.22}
T:\,\bigl[H^{(l_{0})}(\Gamma,r,\theta),\,
H^{(l_{1})}(\Gamma,r,\theta)\,\bigr]_{\psi}\\
\rightarrow\bigl[\bigl(H^{(l_{0})}(\mathbb{R}^{n},r,\theta)\bigl)^{p},\,
\bigl(H^{(l_{1})}(\mathbb{R}^{n},r,\theta)\bigl)^{p}\,\bigr]_{\psi}.
\end{equation}
Here the couples of spaces are normal. Therefore, according to Proposition
\ref{prop2}, the norm of the operator \eqref{f7.22} does not exceed a certain number
$\widetilde{c}=\widetilde{c}(\psi)>0$, which is independent of the parameter $r$,
specifically. Whence, by Proposition \ref{prop3} and property \eqref{f7.19}, we
obtain the bounded operator
\begin{equation}\label{f7.23}
T:\,\bigl[H^{(l_{0})}(\Gamma,r,\theta),\,
H^{(l_{1})}(\Gamma,r,\theta)\,\bigr]_{\psi}\\
\rightarrow\bigl(H^{\eta}(\mathbb{R}^{n},r,\theta)\bigr)^{p}, \quad\mbox{whose
norm}\;\leq\widetilde{c}.
\end{equation}

Along with $T$, consider the linear mapping of "sewing"
$$
K:\,(w_{1},\ldots,w_{p})\mapsto
\sum_{j=1}^{p}\Theta_{j}\bigl((\eta_{j}w_{j})\circ\alpha_{j}^{-1}\bigr),
$$
where $w_{1},\ldots,w_{p}$ are distributions defined in $\mathbb{R}^{n}$. Here the
function $\eta_{j}\in C^{\infty}(\mathbb{R}^{n})$ is equal to $1$ on the set
$\alpha_{j}^{-1}(\mathrm{supp}\,\chi_{j})$ and is compactly supported, whereas
$\Theta_{j}$ denotes the operator of extension by zero from $\Gamma_{j}$ onto
$\Gamma$. We have the bounded operators
\begin{gather}\label{f7.24}
K:\,\bigl(H^{(s)}(\mathbb{R}^{n})\bigr)^{p}\rightarrow H^{(s)}(\Gamma)\quad
\mbox{for each}\quad s\in\mathbb{R},\\
K:\,\bigl(H^{\varphi}(\mathbb{R}^{n})\bigr)^{p}\rightarrow H^{\varphi}(\Gamma)\quad
\mbox{for each}\quad\varphi\in\mathrm{RO}. \label{f7.25}
\end{gather}

Note that the boundedness of the operator \eqref{f7.24} is a known property of
Sobolev spaces (see, e.g., \cite[Sec. 2.6]{Hermander63} or \cite[p.
86]{MikhailetsMurach10}). The boundedness of the operator \eqref{f7.25} follows from
this property with the help of interpolation. Namely, let $\varphi$, $s_{0}$,
$s_{1}$, and $\psi$ be the same as that in Proposition~\ref{prop1}. Then applying
the interpolation with the function parameter $\psi$ to \eqref{f7.24} with
$s\in\{s_{0},s_{1}\}$, we get the boundedness of the operator \eqref{f7.25} by
virtue of Propositions \ref{prop1} and~\ref{prop3}.

Let $c_{1}$ be the maximum of the norms of the operators \eqref{f7.24} and
\eqref{f7.25}, where $s\in\{l_{0},l_{0}-\theta,l_{1},l_{1}-\theta\}$ and
$\varphi\in\{\eta,\eta\varrho^{-\theta}\}$. The number $c_{1}>0$ does not depend on
the parameter $r$. We may directly verify that the norms of the operators
\begin{gather}\label{f7.26}
K:\,\bigl(H^{\eta}(\mathbb{R}^{n},r,\theta)\bigr)^{p}\rightarrow
H^{\eta}(\Gamma,r,\theta),\\
K:\,\bigl(H^{(l_{j})}(\mathbb{R}^{n},r,\theta)\bigr)^{p}\rightarrow
H^{(l_{j})}(\Gamma,r,\theta),\quad j=0,\,1, \label{f7.27}
\end{gather}
does not exceed the number $c_{1}$. Applying the interpolation with the parameter
$\psi$ to \eqref{f7.27}, we obtain the bounded operator
\begin{equation}\label{f7.28}
K:\,\bigl[\,\bigl(H^{(l_{0})}(\mathbb{R}^{n},r,\theta)\bigl)^{p},\,
\bigl(H^{(l_{1})}(\mathbb{R}^{n},r,\theta)\bigr)^{p}\,\bigr]_{\psi}\rightarrow
\bigl[\,H^{(l_{0})}(\Gamma,r,\theta),\,H^{(l_{1})}(\Gamma,r,\theta)\,\bigr]_{\psi}.
\end{equation}
Its norm does not exceed $\widetilde{c}\,c_{1}$ in view of Proposition \ref{prop2}
(note that both couples of spaces are normal in \eqref{f7.28}). Whence, by
\eqref{f7.19} and Proposition \ref{prop3}, we obtain the bounded operator
\begin{gather}\label{f7.29}
K:\,\bigl(H^{\eta}(\mathbb{R}^{n},r,\theta)\bigr)^{p}\rightarrow
\bigl[\,H^{(l_{0})}(\Gamma,r,\theta),\,H^{(l_{1})}(\Gamma,r,\theta)\,\bigr]_{\psi},
\quad\mbox{whose norm}\;\leq \widetilde{c}\,c_{1}.
\end{gather}

By the choice of the functions $\chi_{j}$ and $\eta_{j}$, we may write
\begin{gather*}
KTu=\sum_{j=1}^{p}\,\Theta_{j}\bigl((\eta_{j}\,
((\chi_{j}u)\circ\alpha_{j}))\circ\alpha_{j}^{-1}\bigr)\\=
\sum_{j=1}^{p}\,\Theta_{j}\bigl((\chi_{j}u)\circ\alpha_{j}\circ\alpha_{j}^{-1}\bigr)=
\sum_{j=1}^{p}\,\chi_{j}u=u,
\end{gather*}
that is $KTu=u$ for each $u\in\mathcal{D}'(\Gamma)$. Therefore, multiplying
\eqref{f7.29} by the isometric operator \eqref{f7.20}, we obtain the bounded
identity operator
$$
I=KT:\,H^{\eta}(\Gamma,r,\theta)\rightarrow
\bigl[\,H^{(l_{0})}(\Gamma,r,\theta),\,H^{(l_{1})}(\Gamma,r,\theta)\,\bigr]_{\psi},
$$
whose norm $\leq\widetilde{c}\,c_{1}$. Besides, taking the product of the operators
\eqref{f7.26} and \eqref{f7.23}  (the norm of \eqref{f7.26} does not exceed
$c_{1}$), we get another bounded identity operator
$$
I=KT:\,\bigl[\,H^{(l_{0})}(\Gamma,r,\theta),\,H^{(l_{1})}(\Gamma,r,\theta)\,\bigr]_{\psi}
\rightarrow H^{\eta}(\Gamma,r,\theta),
$$
whose norm $\leq\widetilde{c}\,c_{1}$. These identity operators yield the required
estimate \eqref{f7.16}, where the number $c_{0}:=\widetilde{c}\,c_{1}\geq1$ does not
depend on $r\geq0$ and $u\in H^{\eta}(\Gamma)$..
\end{proof}

Now, applying Lemma \ref{lem3}, we may give

\begin{proof}[Proof of Theorem]
As has been mentioned in Section \ref{sec4}, Theorem is known in the case of Sobolev
inner--product spaces. By using parameter-depended spaces introduced above, we may
reformulate Theorem for the Sobolev scale in the following way. There exists a
number $\lambda_{0}>0$ such that the isomorphism
\begin{equation}\label{f31}
A(\lambda):\,H^{s+mq}(\Gamma,|\lambda|^{q},mq)\,\leftrightarrow\,H^{s}(\Gamma)
\end{equation}
holds for each $s\in\mathbb{R}$ and $\lambda\in K$ with $|\lambda|\geq\lambda_{0}$.
Moreover, the norms of the operator \eqref{f31} and its inverse are uniformly
bounded with respect to $\lambda$.

Let $\varphi\in\mathrm{RO}$, choose numbers $s_{0}<\sigma_{0}(\varphi)$ and
$s_{1}>\sigma_{1}(\varphi)$, and define the interpolation parameter $\psi$ by
\eqref{f11}. Applying the interpolation with this parameter to \eqref{f31} for
$s\in\{s_{0},s_{1}\}$, we obtain the isomorphism
\begin{equation}\label{f32}
A(\lambda):\,\bigl[\,H^{s_{0}+mq}(\Gamma,|\lambda|^{q},mq),\,
H^{s_{1}+mq}(\Gamma,|\lambda|^{q},mq)\,\bigr]_{\psi}\,\leftrightarrow\,
\bigl[\,H^{s_{0}}(\Gamma),\,H^{s_{1}}(\Gamma)\,\bigr]_{\psi}.
\end{equation}
According to Proposition \ref{prop2}, the norms of the operator \eqref{f32} and its
inverse are uniformly bounded with respect to $\lambda$. Now, by Lemma \ref{lem3}
and Proposition \ref{prop1}, we draw a conclusion that \eqref{f32} yields the
isomorphism
\begin{equation}\label{f33}
A(\lambda):\,H^{\varphi\varrho^{mq}}(\Gamma,|\lambda|^{q},mq)\leftrightarrow
H^{\varphi}(\Gamma)
\end{equation}
such that the norms of the operator \eqref{f33} and its inverse are uniformly
bounded with respect to $\lambda$. Here we apply Lemma \ref{lem3} for
$\eta:=\varphi\varrho^{mq}$, $l_{0}:=s_{0}+mq<\sigma_{0}(\eta)$,
$l_{1}:=s_{1}+mq>\sigma_{1}(\eta)$, $\theta:=mq$, and $r:=|\lambda|^{q}$, and we
also note that $\psi$ satisfies \eqref{f7.15}. The isomorphism \eqref{f33} and the
norms property just proved mean in view of \eqref {f7.14} that Theorem is true.
\end{proof}

\end{document}